\documentclass[a4paper]{article}
\usepackage{latexsym}
\usepackage{amsmath}
\usepackage{amssymb}
\usepackage{enumerate}
\usepackage{amsthm} 
\usepackage{graphicx} 
\usepackage{amssymb,latexsym,times}
\usepackage{proof}

\newcommand{\Si}{\Sigma}
\newcommand{\sub}{\subseteq}
\newcommand{\pr}{\textrm{pr}}
\newcommand{\tuple}[1]{\vec{#1}}
\newcommand {\indep}[3] {#2 ~\bot_{#1}~ #3}
\newcommand {\indepc}[2] {#1 ~\bot~ #2}

\newcommand{\Dom}{\textrm{Dom}}

\newcommand{\Var}{\textrm{Var}}
\newcommand{\Fr}{\textrm{Fr}}

\newcommand{\N}{\mathbb{N}}

\newcommand{\Po}{\mathcal{P}}

\newcommand{\M}{\mathcal{M}}

\newcommand{\on}{\exists}
\newcommand{\ja}{\wedge}

\def\dep{=\!\!}

\newcommand{\indlogic}{\FO (\bot_{\rm c})}
\newcommand{\inclogic}{\FO (\subseteq)}
\newcommand{\indNRlogic}{\FO (\bot)}

\newcommand{\deplogic}{\FO (\dep(\ldots ))}

\newcommand{\vdasn}{\vdash^*}
\newcommand{\der}{\Si \vdasn }
\def\dep{=\!\!}
\newcommand{\FO}{{\rm FO}}

\theoremstyle{plain}
\newtheorem{thm}[equation]{Theorem}
\newtheorem{lem}[equation]{Lemma}
\newtheorem{prop}[equation]{Proposition}
\newtheorem{cor}[equation]{Corollary}

\def\ci{\perp\!\!\!\perp}

\theoremstyle{definition}
\newtheorem{maa}[equation]{Definition}

\newtheorem{esim}[equation]{Example}

\begin{document}

\title{A finite axiomatization of conditional independence and inclusion dependencies\thanks{The authors were supported by grant 264917 of the Academy of Finland.}}

\author{Miika Hannula \thanks{Department of Mathematics and Statistics, University of Helsinki, Finland. \texttt{miika.hannula@helsinki.fi}} \and Juha Kontinen\thanks{Department of Mathematics and Statistics, University of Helsinki, Finland. \texttt{juha.kontinen@helsinki.fi}}}

\maketitle
\begin{abstract}
We present a complete finite axiomatization of the unrestricted implication problem for inclusion and conditional independence atoms in the context of dependence logic. For databases, our  result implies a  finite  axiomatization of the unrestricted implication problem for inclusion, functional, and embedded multivalued dependencies in the unirelational case. 
\end{abstract}

\section{Introduction}

We formulate a finite axiomatization of the implication problem for inclusion and conditional independence atoms (dependencies) in the dependence logic context. The input of this problem is given by a finite set  $\Sigma\cup \{\phi\}$ consisting  of conditional independence atoms and inclusion atoms, and the question to decide is whether the following logical consequence holds
\begin{equation}\label{AX}
\Sigma \models \phi.
\end{equation}
Independence logic \cite{gradel10} and inclusion logic \cite{galliani12} are recent variants of dependence logic the semantics of which are defined over sets of assigments (teams) rather than a single assignment as in first-order logic. By viewing a team  $X$  with domain  $\{x_1,\ldots,x_k\}$  as a relation schema $X[\{x_1,\ldots,x_k\}]$, our results provide a finite  axiomatization for the unrestricted implication problem of  inclusion, functional, and embedded multivalued database dependencies over $X[\{x_1,\ldots,x_k\}]$.

Dependence logic \cite{vaananen07}  extends first-order logic by dependence atomic formulas 
\begin{equation}\label{da}\dep(x_1,\ldots,x_n)
\end{equation} the meaning of which is that the value of $x_n$ is functionally determined by the values of $x_1,\ldots, x_{n-1}$. 
Independence logic replaces the dependence atoms by independence atoms 
\[\vec{y}\bot_{\vec{x}} \vec{z},\]
 the intuitive meaning of which is that, with respect to any fixed value of  $ \vec x$, the variables $\vec y$  are totally independent of the variables $\vec z$. Furthermore, inclusion logic is based on inclusion atoms of the form 
\[ \vec{x}\subseteq \vec{y},\]
 with the meaning that all the values of $\vec{x}$ appear also as values for $\vec{y}$. By viewing a team  $X$ of assignments  with domain  $\{x_1,\ldots,x_k\}$  as a relation schema $X[\{x_1,\ldots,x_k\}]$, the atoms $\dep(\vec{x})$,   $\vec{x}\subseteq \vec{y}$, and  $\vec{y}\bot_{\vec{x}} \vec{z}$ correspond to functional, inclusion, and embedded multivalued database dependencies. Furthermore,  the atom $\dep(x_1,\ldots,x_n)$ can be alternatively expressed as  \[x_n\bot_{x_1\ldots x_{n-1}} x_n,\] 
hence our results for independence atoms cover also the case where dependence atoms are present.   

The team semantics of dependence logic is a very flexible logical framework in which various notions of dependence and independence can be formalized.  Dependence logic and its variants  have  turned out to be applicable in various areas. For example, V\"a\"an\"anen and Abramsky have recently  axiomatized and formally proved Arrow's Theorem from social choice theory and, certain No-Go theorems from the foundations of quantum mechanics  in the context of independence logic \cite{AV}. Also, the pure independence atom $\vec{y}\bot \vec{z}$ and its axioms has various concrete interpretations such as independence   $X\ci Y$  between two sets of random variables  \cite{geiger91}, and  independence in vector spaces and algebraically closed fields  \cite{GP}. 

Dependence logic is equi-expressive with existential second-order logic (ESO). Furthermore, the  set of valid formulas of dependence logic has the same complexity as that of full second-order logic,  hence it is not possible to  give a complete axiomatization of dependence logic \cite{vaananen07}. However,  by restricting attention to syntactic fragments  \cite{FYJV, hannula13,kontinenv11} or by modifying the semantics \cite{btxdoc} complete axiomatizations have recently been obtained.  
The axiomatization presented in this article   is based on the classical characterization of logical implication between dependencies in terms of  the \emph{Chase} procedure \cite{Maier:1979:TID:320107.320115}. The novelty in our approach is the use of the so-called \emph{Lax} team semantics  of independence logic to simulate the chase on the logical level using only inclusion and independence atoms and existential quantification.  

In database theory, the implication problems of various types of database dependencies have been extensively studied starting from Armstrong's axiomatization for functional dependencies \cite{armstrong74}. Inclusion dependencies were  axiomatized in \cite{casanova83}, and an axiomatization for pure independence atoms is also known (see \cite{DBLP:journals/jcss/Paredaens80,geiger91,DBLP:conf/wollic/KontinenLV13}). On the other hand, the implication problem of embedded multivalued dependencies, and of inclusion dependencies and functional dependencies together, are  known to be undecidable \cite{Herrmann1995221, Herrmann2006, chandra85}, hence simple axiomatization (that would yield a decision procedure) is deemed impossible. On the other hand, the unrestricted implication problem of inclusion and functional dependencies has been finitely axiomatized in \cite{mitchell83} using a so-called \emph{Attribute Introduction Rule} that allows new attribute names representing derived attributes to be introduced into deductions. These new attributes can be thought of as implicitly existentially quantified. Our \emph{Inclusion Introduction Rule} is essentially equivalent to the Attribute Introduction Rule of \cite{mitchell83}.  It is also worth noting that the chase procedure has been used to axiomatize the unrestricted implication problem of various classes of dependencies, e.g., \emph{Template Dependencies} \cite{DBLP:journals/jacm/SadriU82}, and  \emph{Typed Dependencies} \cite{DBLP:journals/siamcomp/BeeriV84}. Finally we note that the role of inclusion atom in our axiomatization has some similarities to the axiomatization of the class of \emph{Algebraic Dependencies} \cite{DBLP:journals/jcss/YannakakisP82}.

\section{Preliminaries}
In this section we define team semantics and introduce dependence, independence and inclusion atoms. The version of team semantics presented here is the Lax one, originally introduced in \cite{galliani12}, which will turn out to be valuable for our purposes due to its interpretation of existential quantification.
\subsection{Team semantics}
The semantics is formulated using sets of assignments called teams instead of single assignments. Let $\M$ be a model with domain $M$. An \emph{assignment} $s$ of $\M$ is a finite mapping from a set of variables into $M$. A \emph{team} $X$ over $\M$ with domain $\Dom(X) = V$ is a set of assignments from $V$ to $M$. For a subset $W$ of $V$, we write $X \upharpoonright W$ for the team obtained by restricting all the assignments of $X$ to the variables in $W$.

If $s$ is an assignment, $x$ a variable, and $a \in A$, then $s[a/x]$ denotes the assignment (with domain $\Dom(s) \cup \{x\}$) that agrees with $s$ everywhere except that it maps $x$ to $a$. For an assignment $s$, and a tuple of variables $\vec{x} = (x_1,...,x_n)$, we sometimes denote the tuple $(s(x_1), . . . , s(x_n))$ by $s(\vec{x})$. For a formula $\phi$, $\Var(\phi)$ and $\Fr(\phi)$ denote the sets of variables that appear in $\phi$ and appear free in $\phi$, respectively. For a finite set of formulas $\Si = \{\phi_1, \ldots ,\phi_n\}$, we write $\Var(\Si)$ for $\Var(\phi_1) \cup \ldots \cup \Var(\phi_n)$, and define $\Fr(\Si)$ analogously. When using set operations  $\vec{x} \cup \vec{y}$ and $\vec{x}\setminus \vec{y} $  for sequences of variables $\vec{x}$ and $\vec{y}$, then these sequences are interpreted as the sets of elements of these sequences. 

Team semantics is defined for first-order logic formulas as follows:


\begin{maa}[Team semantics]\label{team}
Let $\M$ be a model and let $X$ be any team over it. Then 
\begin{itemize}
\item If $\phi$ is a first-order atomic or negated atomic formula, then $\M \models_X \phi$ if and only if for all $s \in X$, $\M \models_s \phi$ (in Tarski semantics).
\item $\M \models_X \psi \vee \theta$ if and only if there are $Y$ and $Z$ such that $X = Y \cup Z$ and $\M \models_Y \psi$ and $\M \models_Z \theta$.
\item $\M \models_X \psi \wedge \theta$ if and only if $\M \models_X \psi$ and $\M \models_X \theta$.
\item $\M \models_X \exists v \psi$ if and only if there is a function $F : X \rightarrow \mathcal P(M) \backslash \{\emptyset\}$ such that $\M \models_{X[F/v]} \psi$, where $X[F/v] = \{s[m/v] : s \in X, m \in F(s)\}$.
\item $\M \models_X \forall v \psi$ if and only if $\M \models_{X[M/v]} \psi$, where $X[M/v] = \{s[m/v] : s \in X, m \in M\}$.
\end{itemize}
\end{maa}
The following lemma is an immediate consequence of Definition \ref{team}.
\begin{lem}\label{ketjulemma}
Let $\M$ be a model, $X$ a team and $\on x_1 \ldots \on x_n \phi$ a formula in team semantics setting where $x_1, \ldots ,x_n$ is a sequence of 
variables. Then
\begin{equation*}\M \models_X \on x_1 \ldots \on x_n \phi\textrm{ iff for some function }F:X\rightarrow \Po(M^n)\setminus\{\emptyset\}\textrm{, }\M \models_{X[F/x_1\ldots x_n]} \phi
\end{equation*} 
where $X[F/x_1\ldots x_n]:= \{s[a_1/x_1]\ldots[a_n/x_n] \mid (a_1, \ldots ,a_n) \in F(s)\}$.
\end{lem}
If $\M \models_X \phi$, then we say that $X$ \emph{satisfies} $\phi$ in $\M$.
If $\phi$ is a sentence (i.e. a formula with no free variables), then we say that $\phi$ is \emph{true} in $\M$, and write $\M \models \phi$, if $\M \models_{\{\emptyset\}} \phi$ where $\{\emptyset\}$ is the team consisting of the empty assignment. Note that $\{\emptyset\}$ is different from the \emph{empty team} $\emptyset$ containing no assignments.

In the team semantics setting, formula $\psi$ is a \emph{logical consequence} of $\phi$, written $\phi\Rightarrow \psi$, if for all models $\M$ and teams $X$, with $\Fr(\phi)\cup \Fr(\psi)\subseteq \Dom(X)$, 
$$\M\models_X\phi \Rightarrow   \M\models_X\psi.$$
Formulas $\phi$ and $\psi$ are said to be \emph{logically equivalent} if $\phi \Rightarrow \psi$ and $\psi \Rightarrow \phi$.
Logics $\mathcal{L}$ and  $\mathcal{L}'$ are said to be equivalent, $\mathcal{L}=\mathcal{L}'$, if  every $\mathcal{L}$-sentence $\phi$ is equivalent to some $\mathcal{L}'$-sentence $\psi$, and vice versa. 


\subsection{Dependencies in team semantics}
Dependence, independence and inclusion atoms are given the following semantics. 
\begin{maa}
Let $\tuple x$ be a tuple of variables and $y$ a variable. Then $\dep(\tuple x ,y)$ is a \emph{dependence atom} with the semantic rule
\begin{itemize}
\item $\M \models_X \dep(\tuple x ,y)$ if and only if for any $s, s' \in X$ with $s(\vec{x})=s'(\vec{x})$, $s(y)=s'(y)$.
\end{itemize}
Let $\tuple x$, $\tuple y$ and $\tuple z$ be tuples of variables. Then $\indep{\tuple x}{\tuple y}{\tuple z}$ is a \emph{conditional independence atom} with the semantic rule
\begin{itemize}
\item $\M \models_X \indep{\tuple x}{\tuple y}{\tuple z}$ if and only if for any $s, s' \in X$ with $s(\vec{x})=s'(\vec{x})$ there is a $s'' \in X$ such that $s''(\vec{x})=s(\vec{x})$, $s''(\vec{y})=s(\vec{y})$ and $s''(\vec{z})=s'(\vec{z})$.
\end{itemize}
Furthermore, we will write $\indepc{\tuple x}{\tuple y}$ as a shorthand for $\indep{\emptyset}{\tuple x}{\tuple y}$, and call it a \emph{pure independence atom}.

Let $\tuple x$ and $\tuple y$ be two tuples of variables of the same length. Then $\tuple x \subseteq \tuple y$
is an \emph{inclusion atom} with the semantic rule 
\begin{itemize}
\item $\M \models_X \tuple x \subseteq \tuple y$ if and only if for any $s \in X$ there is a $s'\in X$ such that $s(\vec{x})=s'(\vec{y})$.
\end{itemize}
\end{maa}
Note that in the definition of an inclusion atom $\tuple x \subseteq \tuple y$, the tuples $\tuple x$ and $\tuple y$ may both have repetitions. Also in the definition of a conditional independence atom $\indep{\tuple x}{\tuple y}{\tuple z}$, the tuples $\tuple x$, $\tuple y$ and $\tuple z$ are not necessarily pairwise disjoint. Thus any dependence atom $\dep(\vec{x},y)$ can be expressed as a conditional independence atom $\indep{\vec{x}}{y}{y}$. Also any independence atom $\indep{\vec{x}}{\vec{y}}{\vec{z}}$ can be expressed as a conjunction of dependendence atoms and an independence atom $\indep{\vec{x}}{\vec{y}^*}{\vec{z}^*}$ where $\vec{x}$, $\vec{y}^*$ and $\vec{z}^*$ are pairwise disjoint. For disjoint tuples $\tuple x$, $\tuple y$ and $\tuple z$, independence atom $\indep{\vec{x}}{\vec{y}}{\vec{z}}$ corresponds to the embedded multivalued dependency $\vec{x} \twoheadrightarrow \vec{y} |\vec{z}$. Hence the class of conditional independence atoms corresponds to the class of functional dependencies and embedded multivalued dependencies in database theory.

\begin{prop}[\cite{galhankon13}]
Let $\indep{\vec{x}}{\vec{y}}{\vec{z}}$ be a conditional independence atom where $\vec{x}$, $\vec{y}$ and $\vec{z}$ are tuples of variables. If $\vec{y}^*$ lists the variables in $\vec{y} - \vec{x}\cup\vec{z}$, $\vec{z}^*$ lists the variables in $\vec{z}-\vec{x}\cup\vec{y}$, and $\vec{u}$ lists the variables in $\vec{y} \cap \vec{z}-\vec{x}$, then
$$\M \models_X \indep{\vec{x}}{\vec{y}}{\vec{z}} \Leftrightarrow \M \models_X \indep{\vec{x}}{\vec{y}^*}{\vec{z}^*} \ja \bigwedge_{u\in \vec{u}} \dep(\vec{x},u).$$
\end{prop}

The extension of first-order logic by dependence atoms, conditional independence atoms and inclusion atoms is called \emph{dependence logic} ($\deplogic$), \emph{independence logic} ($\indlogic$) and \emph{inclusion logic} ($\inclogic$), respectively. The fragment of independence logic containing only pure independence atoms is called \emph{pure independence logic}, written $\indNRlogic$. For a collection of atoms $\mathcal C \subseteq \{=\!\!(\ldots), \bot_{\rm c}, \subseteq\}$, we will write $\FO(\mathcal C)$ (omitting the set parenthesis of $\mathcal{C}$) for first-order logic with these atoms.

We end this section with a list of properties of these logics.
\begin{prop}\label{PROP}
For $\mathcal C = \{=\!\!(\ldots), \bot_{\rm c}, \subseteq\}$, the following hold.
\begin{enumerate}
\item(Empty Team Property)\label{thm:etp}
For all models $\M$ and formulas $\phi \in \FO(\mathcal C)$ 
$$\M \models_\emptyset \phi.$$

\item(Locality \cite{galliani12})\label{thm:loc}
If $\phi \in \FO(\mathcal C)$ is such that $\Fr (\phi)\sub V$, then for all models $\M$ and teams $X$, 
$$\M \models_X \phi \Leftrightarrow \M \models_{X\upharpoonright V} \phi.$$
\item \label{thm:inc_ind}\cite{galliani12}
An inclusion atom $\tuple x \subseteq \tuple y$ is logically equivalent to the pure independence logic formula
\[
\forall v_1 v_2 \tuple z ((\tuple z \not = \tuple x \wedge \tuple z \not = \tuple x) \vee (v_1 \not = v_2 \wedge \tuple z \not = \tuple y) \vee ((v_1 = v_2 \vee \tuple z = \tuple y) \wedge \indepc{\tuple z}{v_1 v_2}))
\]
where $v_1$, $v_2$ and $\tuple z$ are new variables.
\item\cite{vaananen13} 
Any independence logic formula is logically equivalent to some pure independence logic formula.
\item \cite{vaananen07,gradel10} Any dependence (or independence) logic sentence $\phi$ is logically equivalent to some existential second-order sentence $\phi^*$, and vice versa.
\item\label{thm:ILGFP}\cite{gallhella13} Any inclusion logic sentence $\phi$ is logically equivalent to some positive greatest fixpoint logic sentence $\phi^*$, and vice versa.
\end{enumerate}
\end{prop}


\section{Deduction system}

In this section we present a sound and complete axiomatization for the implication problem of inclusion and independence atoms. The implication problem is given by a finite set $\Si \cup \{\phi\}$ consisting of conditional independence and inclusion atoms, and the question is to decide whether
$\Si \models \phi.$
\begin{maa}\label{axioms}
In addition to the usual introduction and elimination rules for conjunction, we adopt the following rules for conditional independence and inclusion atoms.
\begin{enumerate}
\item Reflexivity:   $$\vec{x} \subseteq \vec{x}.$$
\item Projection and Permutation: $$\textrm{if }x_1 \ldots x_n \subseteq y_1 \ldots y_n  \textrm{, then } x_{i_1} \ldots x_{i_k} \subseteq y_{i_1} \ldots y_{i_k},$$ for each sequence $i_1 , \ldots ,i_k$ of integers from $\{1, \ldots ,n\}$.
\item Transitivity: $$\textrm{if }\vec{x} \subseteq \vec{y}\ja \vec{y} \subseteq \vec{z} \textrm{, then }\vec{x} \subseteq \vec{y}.$$
\item Identity Rule: \[\textrm{if }ab\sub cc \ja \phi \textrm{, then } \phi',\]
where $\phi'$ is obtained from $\phi$ by replacing any number of occurrences of $a$ by $b$.
\item Inclusion Introduction: $$\textrm{if }\vec{a} \subseteq \vec{b} \textrm{, then } \vec{a}x \subseteq \vec{b} c,$$ where $x$ is a \emph{new} variable.
\item Start Axiom: $$\vec{a} \vec{c} \subseteq \vec{a} \vec{x} \ja \indep{\vec{a}}{\vec{b}}{\vec{x}} \ja \vec{a} \vec{x} \subseteq \vec{a} \vec{c}$$ where $\vec{x}$ is a sequence of pairwise distinct \emph{new} variables.
\item Chase Rule: $$\textrm{if }\indep{\vec{x}}{\vec{y}}{\vec{z}} \ja \vec{a}\vec{b} \subseteq \vec{x}\vec{y} \ja \vec{a}\vec{c} \subseteq \vec{x}\vec{z}  \textrm{, then } \vec{a}\vec{b}\vec{c} \subseteq \vec{x}\vec{y}\vec{z}.$$
\item Final Rule: $$\textrm{if }\vec{a} \vec{c} \subseteq \vec{a} \vec{x}\ja \indep{\vec{a}}{\vec{b}}{\vec{x}} \ja \vec{a}\vec{b} \vec{x} \subseteq \vec{a}\vec{b} \vec{c}  \textrm{, then } \indep{\vec{a}}{\vec{b}}{\vec{c}}.$$
\end{enumerate}
\end{maa}
In an application of Inclusion Introduction, the variable $x$ is called the new variable of the deduction step. Similarly, in an application of Start Axiom, the variables of $\vec{x}$ are called the new variables of the deduction step.
A deduction from $\Si$ is a sequence of formulas $(\phi_1, \ldots ,\phi_n)$ such that:
\begin{enumerate}
\item Each $\phi_i$ is either an element of $\Si$, an instance of Reflexivity or Start Axiom, or follows from one or more formulas of $\Si \cup \{\phi_1, \ldots ,\phi_{i-1}\}$ by one of the rules presented above.
\item If $\phi_i$ is an instance of Start Axiom (or follows from $\Si \cup \{\phi_1, \ldots ,\phi_{i-1}\}$ by Inclusion Introduction), then the new variables of $\vec{x}$ (or the new variable $x$) must not appear in $\Si \cup \{\phi_1,\ldots ,\phi_{i-1}\}$.
\end{enumerate}
We say that $\phi$ is provable from $\Si$, written $\Si \vdash \phi$, if there is a deduction $(\phi_1, \ldots ,\phi_n)$ from $\Si$ with $\phi=\phi_n$ and such that no variables in $\phi$ are new in $\phi_1, \ldots ,\phi_n$. 

\section{Soundness}
First we prove the soundness of these axioms. Identity Rule and Start Axiom are sound if we interpret all the new variables as existentially quantified.

\begin{lem}\label{fakesound}
Let $(\phi_1, \ldots ,\phi_n)$ be a deduction from $\Si$, and let $\vec{y}$ list 
all the new variables of the deduction steps. Let $\M$ and $X$ be such that $\M \models_X \Si$ and $\Var (\Si_n) \setminus\vec{y} \sub \Dom(X)$ where $\Si_n := \Si \cup\{\phi_1, \ldots ,\phi_{n}\}$. Then
 $$\M \models_X \on \vec{y} \bigwedge \Si_n.$$
\end{lem}

\begin{proof}
We show the claim by induction on $n$. So assume that the claim holds for any deduction of length $n$. We prove that the claim holds for deductions of lenght $n+1$ also. Let $(\phi_1, \ldots ,\phi_{n+1})$ be a deduction from $\Si$, and let $\vec{y}$ and $\vec{z}$ list all the new variables of the deduction steps $\phi_1, \ldots ,\phi_n$ and $\phi_{n+1}$, respectively. Note that $\phi_{n+1}$ might not contain any new variables in which case $\vec{z}$ is empty. Assume that $\M \models_X \Si$ for some $\M$ and $X$, where $\Var (\Si_{n+1}) \setminus\vec{y}\vec{z} \sub \Dom(X)$. By Proposition \ref{PROP}.\ref{thm:loc} we may assume that $\Var (\Si_{n+1}) \setminus\vec{y}\vec{z} = \Dom(X)$.
 We need to show that
$$\M \models_X \on \vec{y}\on \vec{z} \bigwedge\Sigma_{n+1}.$$
 By the induction assumption,
$$\M \models_X \on \vec{y} \bigwedge\Sigma_n$$
when by Lemma \ref{ketjulemma} there is a function $F:X \rightarrow \Po(M^{|\vec{y}|})\setminus\{\emptyset\}$ such that 
\begin{equation}\label{heipähei}
\M \models_{X'} \bigwedge \Sigma_n
\end{equation}
where $X':= X[F/\vec{y}]$. It suffices to show that
$$\M \models_{X'} \on \vec{z} \bigwedge \Sigma_{n+1}.$$
If $\phi_{n+1}$ is an instance of Start Axiom, or follows from $\Si_n$ by Inclusion Introduction, then by Lemma \ref{ketjulemma} it suffices to find a $G: X' \rightarrow \Po(M^{|\vec{z}|})\setminus\{\emptyset\}$, such that $\M \models_{X'[G/\vec{z}]} \phi_{n+1}$. For this note that no variable of $\vec{z}$ is in $\Var (\Si_n)$, and hence by Proposition \ref{PROP}.\ref{thm:loc} $\M \models_{X'[G/\vec{z}]} \Si_n$ follows from \eqref{heipähei}. Otherwise, if $\vec{z}$ is empty, then it suffices to show that $\M \models_{X'} \phi_{n+1}$. 

The cases where $\phi_{n+1}$ is an instance of Reflexivity, or follows from $\Si_n$ by a conjunction rule,  Projection and Permutation, Transitivity or Identity are straightforward. We prove the claim in the cases where one of the last four rules is applied.
\begin{itemize}
\item Inclusion Introduction: Then $\phi_{n+1}$ is of the form $\vec{a}x \subseteq \vec{b} c$ where $\vec{a} \subseteq \vec{b}$ is in $\Si_n$. Let $s \in X'$. Since $\M \models_{X'} \vec{a} \subseteq \vec{b}$ there is a $s'\in X'$ such that $s(\vec{a})=s'(\vec{b})$. We let $G(s)=\{s'(c)\}$. Since $x \not\in \Dom(X')$ we conclude that $ \M \models_{X'[G/x]} \vec{a}x \subseteq \vec{b} c$.

\item Start Axiom: Then $\phi_{n+1}$ is of the form $\vec{a} \vec{c} \subseteq \vec{a} \vec{x} \ja \indep{\vec{a}}{\vec{b}}{\vec{x}} \ja \vec{a} \vec{x} \subseteq \vec{a} \vec{c}$. We define $G: X' \rightarrow \Po(M^{|\vec{x}|})\setminus\{\emptyset\}$ as follows:
$$G(s) = \{s'(\vec{c}) \mid s'\in X', s'(\vec{a})=s(\vec{a})\}.$$
Again, since $\vec{x}$ does not list any of the variables in $\Dom(X')$, it is straightforward to show that $$\M \models_{X'[G/\vec{x}]}  \vec{a} \vec{c} \subseteq \vec{a} \vec{x} \ja \indep{\vec{a}}{\vec{b}}{\vec{x}} \ja \vec{a} \vec{x} \subseteq \vec{a} \vec{c}.$$
\item Chase Rule: Then $\phi_{n+1}$ is of the form $\vec{a}\vec{b}\vec{c} \subseteq \vec{x}\vec{y}\vec{z}$ where 
$$\indep{\vec{x}}{\vec{y}}{\vec{z}}\ja\vec{a}\vec{b}\subseteq \vec{x}\vec{y}\ja\vec{a}\vec{c}\subseteq \vec{x}\vec{z} \in \Si_n.$$ Let $s \in X'$. Since $\M \models_{X'}  \vec{a}\vec{b} \subseteq \vec{x}\vec{y} \ja \vec{a}\vec{c} \subseteq \vec{x}\vec{z}$ there are $s',s''\in X'$ such that $s'(\vec{x}\vec{y})=s(\vec{a}\vec{b})$ and $s''(\vec{x}\vec{z} )=s(\vec{a}\vec{c})$. Since $s'(\vec{x})=s''(\vec{x})$ and $\M \models_{X'} \indep{\vec{x}}{\vec{y}}{\vec{z}}$, there is a $s_0\in X'$ such that $s_0(\vec{x}\vec{y}\vec{z})=s(\vec{a}\vec{b}\vec{c})$ which shows the claim.

\item Final Rule: Then $\phi_{n+1}$ is of the form $\indep{\vec{a}}{\vec{b}}{\vec{c}}$ where $$\vec{a} \vec{c} \subseteq \vec{a} \vec{x}\ja\indep{\vec{a}}{\vec{b}}{\vec{x}}\ja\vec{a}\vec{b} \vec{x} \subseteq \vec{a}\vec{b} \vec{c}  \in \Si_n.$$
Let $s,s' \in X'$ be such that $s(\vec{a})=s'(\vec{a})$. Since $\M \models_{X'} \vec{a} \vec{c} \subseteq \vec{a} \vec{x}$ there is a $s_0 \in X'$ such that $s'(\vec{a}\vec{c})=s_0(\vec{a}\vec{x})$. Since $\M\models_{X'} \indep{\vec{a}}{\vec{b}}{\vec{x}}$ and $s(\vec{a})=s_0(\vec{a})$ there is a $s_1\in X'$ such that $s_1(\vec{a}\vec{b}\vec{x})=s(\vec{a}\vec{b})s_0(\vec{x})$. And since $\M \models_{X'} \vec{a}\vec{b} \vec{x} \subseteq \vec{a}\vec{b} \vec{c}$ there is a $s''\in X'$ such that $s''(\vec{a}\vec{b}\vec{c})=s_1(\vec{a}\vec{b}\vec{x})$. Then $s''(\vec{a}\vec{b}\vec{c})=s(\vec{a}\vec{b})s'(\vec{c})$ which shows the claim and concludes the proof.

\end{itemize}
\end{proof}
This gives us the following soundness theorem.
\begin{thm}\label{sound}
Let $\Si \cup \{\phi\}$ be a finite set of conditional independence and inclusion atoms. Then $\Si \models \phi$ if $\Si \vdash \phi$.
\end{thm}
\begin{proof}
Assume that $\Si \vdash \phi$. Then there is a deduction $(\phi_1, \ldots ,\phi_n)$ from $\Si$ such that $\phi = \phi_n$ and no variables in $\phi$ are new in $\phi_1, \ldots ,\phi_n$. Let $\M$ and $X$ be such that $\Var(\Si \cup \{\phi\}) \sub \Dom(X)$ and $\M \models_X \Si $. We need to show that $\M \models_X \phi$. Let $\vec{y}$ list all the new variables in $\phi_1, \ldots ,\phi_n$, and let $\vec{z}$ list all the variables in $\Var(\Si_n)\setminus \vec{y}$ which are not in $\Dom (X)$. We first let $X':=X[\vec{0}/\vec{z}]$ for some dummy sequence $\vec{0}$ when by Theorem \ref{PROP}.\ref{thm:loc}, $\M \models_{X'} \Si$. Then by Theorem \ref{fakesound}, $\M \models_{X'} \on \vec{y} \bigwedge \Si_n$ implying there exists a $F: X' \rightarrow \Po(M^{|\vec{y}|})\setminus\{\emptyset\}$ such that $\M \models_{X''} \phi$, for $X'':=X'[F/\vec{y}]$. Since $X''= X[\vec{0}/\vec{z}][F/\vec{y}]$ and no variables of $\vec{y}$ or $\vec{z}$ appear in $\phi$, we conclude by Theorem \ref{PROP}.\ref{thm:loc} that $\M \models_X \phi$.

\end{proof}
\section{Completeness}
In this section we will prove that the set of axioms and rules presented in Definition \ref{axioms} is complete with respect to the implication problem for conditional independence and inclusion atoms. For this purpose we introduce a graph characterization for the implication problem in subsection \ref{graphchar}. This characterization is based on the classical characterization of the implication problem for various database dependencies using the chase procedure \cite{Maier:1979:TID:320107.320115}. The completeness proof is presented in subsection \ref{comp.proof}.

\subsection{Graph characterization}\label{graphchar}
We will consider graphs consisting of vertices and edges labeled by (possibly multiple) pairs of variables. The informal meaning is that a vertice will correspond to an assignment of a team, and an edge between $s$ and $s'$, labeled by $uw$, will express that $s(u)=s'(w)$. The graphical representation of the chase procedure is adapted from \cite{naumovre}.
\begin{maa}\label{verkko}
Let $G=(V,E)$ be a graph where $E$ consists of non-directed labeled edges $(u,w)_{ab}$ where $ab$ is a pair of variables, and for every pair $(u,w)$ of vertices there can be several $ab$ such that $(u,w)_{ab} \in E$. Then we say that $u$ and $w$ are $ab$-connected, written $u \sim_{ab} w$, if $u=w$ and $a=b$, or if there are vertices $v_0, \ldots ,v_n$ and variables $x_0, \ldots ,x_{n}$ such that $$(u,v_0)_{ax_0},(v_0,v_1)_{x_0x_1}, \ldots ,(v_{n-1},v_n)_{x_{n-1}x_n},(v_n,w)_{x_nb} \in E.$$
\end{maa}
Next we define a graph $G_{\Si ,\phi}$ in the style of Definition \ref{verkko} for a set $\Si \cup \{\phi\}$ of conditional independence and inclusion atoms.
\begin{maa}
Let $\Si\cup\{\phi\}$ be a finite set of conditional independence and inclusion atoms. We let $G_{\Si ,\phi}:=(\bigcup_{n\in\N} V_n,\bigcup_{n\in\N} E_n)$ where $G_n=(V_n,E_n)$ is defined as follows:
\begin{itemize}
\item If $\phi$ is $\indep{\vec{a}}{\vec{b}}{\vec{c}}$, then $V_0:=\{v^+,v^-\}$ and $E_0:=\{(v^+,v^-)_{aa} \mid a \in \vec{a}\}$. If $\phi$ is $\vec{a} \subseteq \vec{b}$, then $V_0:= \{v\}$ and $E_0:= \emptyset$.
\item Assume that $G_n$ is defined. Then for every $v \in V_n$ and $x_1 \ldots x_k \subseteq y_1 \ldots y_k \in \Si$ we introduce a new vertex $v_{\rm new}$ and new edges $(v,v_{\rm new})_{x_iy_i}$, for $1 \leq i \leq k$. Also for every $u,w \in V_n$, $u\neq w$, and $\indep{\vec{x}}{\vec{y}}{\vec{z}} \in \Si$ where $u \sim_{xx} w$, for $x \in \vec{x}$, we introduce a new vertex $v_{\rm new}$ and new edges $(u,v_{\rm new})_{yy}$, $(w,v_{\rm new})_{zz}$, for $y \in \vec{x}\vec{y}$ and $z \in \vec{x}\vec{z}$. We let $V_{n+1}$ and $E_{n+1}$ be obtained by adding these new vertices and edges to the sets $V_n$ and $E_n$.
\end{itemize}
Note that $G_{\Si ,\phi}= G_0$ if $\Si = \emptyset$.
\end{maa}
This gives us a characterization of the following form. Instead of writing $\M \models_X \phi$ we will now write $X \models \phi$, since the satisfaction of an atom depends only on the team $X$.
\begin{thm}\label{char}
Let $\Si\cup\{\phi\}$ be a finite set of conditional independence and inclusion atoms. 
\begin{enumerate}
\item If $\phi$ is $a_1 \ldots a_k \subseteq b_1 \ldots b_k$, then $\Si \models \phi \Leftrightarrow \on w \in V_{\Si ,\phi}(v \sim_{a_ib_i}w \textrm{ for all }1 \leq i \leq k$).
\item If $\phi$ is $\indep{\vec{a}}{\vec{b}}{\vec{c}}$, then $\Si \models \phi \Leftrightarrow \on v \in V_{\Si ,\phi}(v^+ \sim_{bb} v  \textrm{ and } v^- \sim_{cc} v\textrm{ for all }b\in\vec{a}\vec{b} \textrm{ and }c\in\vec{a}\vec{c})$.
\end{enumerate}
\end{thm}
\begin{proof}
We deal with cases $1$ and $2$ simultaneously. First we will show the direction from right to left.  So assume that the right-hand side assumption holds. We show that $\Si \models \phi$. Let $X$ be a team such that $X \models \Si$. We show that $X \models \phi$. For this, let $s,s'\in X$ be such that $s(\vec{a})=s'(\vec{a})$. If $\phi$ is $\indep{\vec{a}}{\vec{b}}{\vec{c}}$, then we need to find a $s''$ such that $s''(\vec{a}\vec{b}\vec{c})=s(\vec{a}\vec{b})s'(\vec{c})$. If $\phi$ is $a_1 \ldots a_k \subseteq b_1 \ldots b_k$, then we need to find a $s''$ such that $s(a_1\ldots a_k)=s''(b_1\ldots b_k)$. We will now define inductively, for each natural number $n$, a function $f_n: V_n \rightarrow X$ such that $f_n(u)(x)=f_n(w)(y)$ if $(u,w)_{xy}\in E_n$. This will suffice for the claim as we will later show.
\begin{itemize}
\item Assume that $n=0$.
\begin{enumerate}
\item\label{caseone} If $\phi$ is $a_1 \ldots a_k \subseteq b_1 \ldots b_k$, then $V_0= \{v\}$ and $E_0= \emptyset$, and we let $f_0(v):=s$.
\item\label{casetwo} If $\phi$ is $\indep{\vec{a}}{\vec{b}}{\vec{c}}$, then $V_0=\{v^+,v^-\}$ and $E_0=\{(v^+,v^-)_{aa} \mid a \in \vec{a}\}$. We let $f_0(v^+):=s$ and $f_0(v^-):=s'$. Then $f(v^+)(a)=f(v^-)(a)$, for $a \in \vec{a}$, as wanted.
\end{enumerate}
\item Assume that $n=m+1$, and that $f_m$ is defined so that $f_m(u)(x)=f_m(w)(y)$ if $(u, w)_{xy} \in E_m$. We let $f_{m+1}(u)=f_m(u)$, for $u\in V_m$. Assume that $v_{\rm new}\in V_{m+1}\setminus V_m$ and that there are $u \in V_m$ and $x_1 \ldots x_l\subseteq y_1\ldots y_l \in \Si$ such that 
$(u,v_{\rm new})_{x_iy_i}\in E_{m+1}\setminus E_m$, for $1 \leq i \leq l$. Since $X \models x_1 \ldots x_l\subseteq y_1\ldots y_l $, there is a $s_0\in X$ such that $f_{m+1}(u)(x_i)=s_0(y_i)$, for $1 \leq i \leq l$. We let $f_{m+1}(v_{\rm new}):=s_0$ when $f_{m+1}(u)(x_i)=f_{m+1}(v_{\rm new})(y_i)$, for $1 \leq i \leq l$, as wanted.

Assume then that $v_{\rm new}\in V_{m+1}\setminus V_m$ and that there are $u,w \in V_m$, $u \neq w$, and $\indep{\vec{x}}{\vec{y}}{\vec{z}} \in \Si$ such that $(u,v_{\rm new})_{yy},(w,v_{\rm new})_{zz}\in E_{m+1}\setminus E_m$, for $y\in \vec{x}\vec{y}$ and $z\in\vec{x}\vec{z}$. Then 
$u \sim_{xx} w$ in $G_m$, for $x \in \vec{x}$. This means that there are vertices $v_0, \ldots ,v_n$ and variables $x_0, \ldots ,x_{n}$, for $x \in \vec{x}$, such that $$(u,v_0)_{xx_0},(v_0,v_1)_{x_0x_1}, \ldots ,(v_{n-1},v_n)_{x_{n-1}x_n},(v_n,w)_{x_nx} \in E_m.$$ 
By the induction assumption then 
$$f_m(u)(x)=f_m(v_0)(x_0)=\ldots =f_m(v_n)(x_n)=f_m(w)(x).$$ Hence, since $X\models \indep{\vec{x}}{\vec{y}}{\vec{z}} $, there is a $s_0$ such that $s_0(\vec{x}\vec{y}\vec{z})=f_m(u)(\vec{x}\vec{y})f_m(w)(\vec{z})$. We let $f_{m+1}(v_{\rm new}):=s_0$ and conclude that $f_{m+1}(u)(y)=f_{m+1}(v_{\rm new})(y)$ and $f_{m+1}(w)(z)=f_{m+1}(v_{\rm new})(z)$, for $y\in \vec{x}\vec{y}$ and $z\in\vec{x}\vec{z}$. This concludes the construction.

\end{itemize}
Now, in case \ref{caseone} there is a $v \in V_{\Si ,\phi}$ such that $v^+ \sim_{bb} v $ and $v^- \sim_{cc} v$ for all $b\in\vec{a}\vec{b}$ and $c\in\vec{a}\vec{c}$. Let $n$ be such that each path witnessing this is in $G_n$. We want to show that choosing $s''$ as $f_n(v)$, $s''(\vec{a}\vec{b}\vec{c})=s(\vec{a}\vec{b})s'(\vec{c})$. Recall that $s=f_n(v^+)$ and $s'=f_n(v^-)$. First, let $b \in \vec{a}\vec{b}$. The case where $v=v^+$ is trivial, so assume that $v \neq v^+$ in which case there are vertices $v_0, \ldots ,v_n$ and variables $x_0, \ldots ,x_{n}$ such that 
$$(v^+,v_0)_{bx_0},(v_0,v_1)_{x_0x_1}, \ldots ,(v_{n-1},v_n)_{x_{n-1}x_n},(v_n,v)_{x_nb} \in E_n$$
when by the construction, $f_n(v^+)(b)=f_n(v)(b)$. Analogously $f_n(v^-)(c)=f_n(v)(c)$, for $c \in \vec{c}$, which concludes this case.

In case \ref{casetwo}, $s''$ is found analogously. This concludes the proof of the direction from right to left.

For the other direction, assume that the right-hand side assumption fails in $G_{\Si ,\phi}$. Again, we deal with both cases simultaneously. We will now construct a team $X$ such that $X \models \Si$ and $X \not\models \phi$. We let $X:=\{s_u \mid u \in V_{\Si, \phi}\}$ where each $s_u:\Var(\Si \cup \{\phi\}) \rightarrow \Po (V_{\Si ,\phi})^{|\Var(\Si \cup \{\phi\})|}$ is defined as follows:
$$s_u(x) := \prod_{y \in \Var(\Si \cup \{\phi\})} \{w \in V_{\Si ,\phi} \mid u\sim_{xy}  w\}. $$
We claim that $s_u(x) = s_w(y) \Leftrightarrow u \sim_{xy} w$. Indeed, assume that $u \sim_{xy} w$. If now $v$ is in the set with the index $z$ of the product $s_u(x)$, then $u\sim_{xz} v$. Since $w \sim_{yx} u$, we have that $w \sim_{yz} v$. Thus $v$ is in the set with the index $z$ of the product $s_w(y)$. Hence by symmetry we conclude that $s_u(x)=s_w(y)$. For the other direction assume that $s_u(x) =s_w(y)$. Then consider the set with the index $y$ of the product $s_w(y)$. Since $w \sim_{yy} w$ by the definition, the vertex $w$ is in this set, and thus by the assumption it is in the set with the index $y$ of the product $s_u(x)$. It follows by the definition that $u \sim_{xy} w$ which shows the claim.

Next we will show that $X \models \Si$. So assume that $\indep{\vec{x}}{\vec{y}}{\vec{z}}\in \Si$ and that $s_u,s_w \in X$ are such that $s_u(\vec{x})=s_w(\vec{x})$. We need to find a $s_v \in X$ such that $s_v(\vec{x}\vec{y}\vec{z})=s_u(\vec{x}\vec{y})s_w(\vec{z})$. Since $u \sim_{xx} w$, for $x \in \vec{x}$, there is a $v\in G_{\Si ,\phi}$ such that $(u,v)_{yy},(w,v)_{zz}\in E_{\Si ,\phi}$, for $y\in \vec{x}\vec{y}$ and $z\in \vec{x}\vec{z}$. Then $s_u(\vec{x}\vec{y})=s_v(\vec{x}\vec{y})$ and $s_w(\vec{x}\vec{z})=s_v(\vec{x}\vec{z})$, as wanted. In case $x_1 \ldots x_l \sub y_1 \ldots y_l \in \Si$, $X\models x_1 \ldots x_l \sub y_1 \ldots y_l$ is shown analogously.

It suffices to show that $X \not\models \phi$. Assume first that $\phi$ is $\indep{\vec{a}}{\vec{b}}{\vec{c}}$. Then $s_{v^+}(\vec{a})=s_{v^-}(\vec{a})$, but by the assumption there is no $ v \in V_{\Si ,\phi}$ such that $v^+ \sim_{bb} v $ and $ v^- \sim_{cc} v$ for all $b\in\vec{a}\vec{b} $ and $c\in\vec{a}\vec{c}$. Hence there is no $s_v \in X$ such that $s_{v}(\vec{a}\vec{b})=s_{v^+}(\vec{a}\vec{b})$ and $s_{v}(\vec{a}\vec{c})=s_{v^-}(\vec{a}\vec{c})$ when $X \not\models \indep{\vec{a}}{\vec{b}}{\vec{c}}$. In case $\phi$ is $a_1 \ldots a_k \subseteq b_1 \ldots b_k$, $X \not\models\phi$ is shown analogously.
\end{proof}
\subsection{Completeness proof}\label{comp.proof}
We are now ready to prove the completeness. Let us first define some notation needed in the proof. We will write $x = y$ for syntactical identity, $x \equiv y$ for an atom of the form $xy\subseteq zz$ implying the identity of $x$ and $y$, and $\vec{x}\equiv\vec{y}$ for an conjunction the form $\bigwedge_{i \leq |\vec{x}|} \pr_i(\vec{x})\equiv \pr_i (\vec{y})$. Let $\vec{x}=(x_1, \ldots ,x_n)$ be a sequence listing $\Var(\Si\cup \{\phi\})$. If $\vec{x}_v$ is a vector of length $|\vec{x}|$ (representing vertex $v$ of the graph $G_{\Si ,\phi}$), and $\vec{p}=(x_{i_1}, \ldots ,x_{i_l})$ is a sequence of variables from $\vec{x}$, then we write $\vec{p}_v$ for 
$$ (\pr_{i_1} ( \vec{x}_v ), \ldots ,\pr_{i_l}(\vec{x}_v)).$$
\begin{thm}\label{complete}
Let $\Si \cup \{\phi\}$ be a finite set of conditional independence and inclusion atoms. Then $\Si \vdash \phi$ if $\Si \models \phi$.
\end{thm}
\begin{proof}
Let $\Si$ and $\phi$ be such that $\Si \models \phi$. We will show that $\Si \vdash \phi$. 
We have two cases: either
\begin{enumerate}
\item $\phi$ is $x_{i_1}\ldots x_{i_m} \sub x_{j_1}\ldots x_{j_m}$ and, by Theorem \ref{char}, there is a $ w \in V_{\Si ,\phi}$ such that $v\sim_{x_{i_k}x_{j_k}} w$ for all $1 \leq k \leq m$, or
\item $\phi$ is $\indep{\vec{a}}{\vec{b}}{\vec{c}}$ and, by Theorem \ref{char}, there is a $ v \in V_{\Si ,\phi}$ such that $v^+\sim_{x_ix_i} v$ and $v^-\sim_{x_jx_j} v $ for all $x_i\in\vec{a}\vec{b} $ and $x_j\in\vec{a}\vec{c}$.
\end{enumerate}
Using this we will show how to create a deduction of $\phi$ from $\Si$. We write $\Si \vdasn \psi$ if $\psi$ appears as a step in the deduction. Recall that the new variables introduced in the deduction steps previously must not appear in $\phi$ but may appear in $\psi$. We will introduce for each $u\in V_{\Si,\phi}$ a sequence $\vec{x}_u$
of length $n$ (and possibly with repetitions) such that $\Si \vdasn \vec{x}_u \subseteq \vec{x} $. For each $(u,w)_{x_ix_j}\in E_{\Si ,\phi}$ we will also show that $\Si \vdasn\pr_i(\vec{x}_u) \equiv \pr_j(\vec{x}_w)$. We do this inductively for $V_n$ and $E_n$ as follows:
\begin{itemize}
\item Assume that $n=0$. Then we have two cases:
\begin{enumerate}
\item Assume that $\phi$ is $x_{i_1}\ldots x_{i_m} \sub x_{j_1}\ldots x_{j_m}$ when $V_0:=\{v\}$ and $E_0:=\emptyset$. Then we let $\vec{x}_{v}:=\vec{x}$ in which case we can derive $\vec{x}_v\sub\vec{x}$ by Reflexivity.

\item Assume that $\phi$ is $\indep{\vec{a}}{\vec{b}}{\vec{c}}$ when $V_0:=\{v^+,v^-\}$ and $E_0:=\{(v^+,v^-)_{x_ix_i} \mid x_i \in \vec{a}\}$. First we use Start Axiom to obtain \begin{equation}
\vec{a} \vec{c} \subseteq \vec{a} \vec{c}^* \ja \indep{\vec{a}}{\vec{b}}{\vec{c}^*} \ja \vec{a} \vec{c}^* \subseteq \vec{a} \vec{c}
\end{equation}
where $\vec{c}^*$ is a sequence of pairwise distinct new variables. Then using Inclusion Introduction and Projection and Permutation we may deduce 
\begin{equation}\label{inkl5}\vec{a}\vec{b}^*\vec{c}^*\vec{d}^*\sub \vec{a}\vec{b}\vec{c}\vec{d}
\end{equation} 
from $\vec{a}\vec{c}^* \sub \vec{a}\vec{c}$ where $\vec{d}$ lists $\vec{x} \setminus \vec{a}\vec{b}\vec{c}$ and $\vec{b}^*\vec{c}^*\vec{d}^*$ is a sequence of pairwise distinct new variables. By Projection and Permutation and Identity Rule we may assume that $\vec{a}\vec{b}^*\vec{c}^*\vec{d}^*$ has repetitions 
exactly where $\vec{a}\vec{b}\vec{c}\vec{d}$ has.
Therefore we can list the variables of $\vec{a}\vec{b}^*\vec{c}^*\vec{d}^*$ in a sequence $\vec{x}_{v^-}$ of length $|\vec{x}|$ where $$\vec{a}\vec{b}^*\vec{c}^*\vec{d}^*= (\pr_{i_1}(\vec{x}_{v^-}), \ldots ,\pr_{i_l}(\vec{x}_{v^-})),$$ for $\vec{a}\vec{b}\vec{c}\vec{d}=(x_{i_1}, \ldots ,x_{i_l})$. Then $\vec{a}_{v^-}\vec{b}_{v^-}\vec{c}_{v^-}\vec{d}_{v^-}= \vec{a}\vec{b}^*\vec{c}^*\vec{d}^*$, and we can derive $ \vec{x}_{v^-} \sub \vec{x}$ from \eqref{inkl5} by Projection and Permutation. We also let $\vec{x}_{v^+}:=\vec{x}$ when $\vec{a}_{v^+}\vec{b}_{v^+}\vec{c}_{v^+}\vec{d}_{v^+}= \vec{a}\vec{b}\vec{c}\vec{d}$. Then $\vec{a}_{v^+}\equiv \vec{a}_{v^-}$ and $\vec{x}_{v^+} \sub \vec{x}$ are derivable by Reflexivity which concludes the case $n=0$.

\end{enumerate}

\item Assume that $n=m+1$ and for each $u \in V_{m}$ there is a sequence $\vec{x}_u$ such that 
$\der \vec{x}_u \subseteq \vec{x}$ and for each $(u,w)_{x_ix_j}\in E_{m}$ also $\der \pr_i(\vec{x}_u) \equiv \pr_j(\vec{x}_w)$. Assume that $v_{\rm new}\in V_{m+1}\setminus V_m$ is such that there are $u \in V_m$ and $x_{i_1} \ldots x_{i_l} \subseteq x_{j_i} \ldots x_{j_l}   \in\Si$ for which we have added new edges $(u,v_{\rm new})_{x_{i_k}x_{j_k}}$ to $V_{m+1}$, for $1 \leq k \leq l$. We will introduce a sequence $\vec{x}_{v_{\rm new}}$ such that $\der \vec{x}_{v_{\rm new}}\sub \vec{x}$ and $\der \pr_{i_k}(\vec{x}_{u})\equiv \pr_{j_k}(x_{v_{\rm new}})$, for $1 \leq k \leq l$.

By Projection and Permutation we deduce first 
\begin{equation}\label{eka}\pr_{i_1}(\vec{x}_u) \ldots \pr_{i_l}(\vec{x}_u)   \sub x_{i_1} \ldots x_{i_l}\end{equation}
from $\vec{x}_u \sub \vec{x}$. Then we obtain 
\begin{equation}\label{blah}\pr_{i_1}(\vec{x}_u) \ldots \pr_{i_l}(\vec{x}_u)  \subseteq x_{j_i} \ldots x_{j_l}
\end{equation}
from \eqref{eka} and $x_{i_1} \ldots x_{i_l} \subseteq x_{j_i} \ldots x_{j_l}$ by Transitivity.

Then by Reflexivity we may deduce $ \pr_{i_1}(\vec{x}_u)    \sub \pr_{i_1}(\vec{x}_u)$ from which we derive by Inclusion Introduction 
\begin{equation}\label{plöh} \pr_{i_1}(\vec{x}_u)y_1    \sub \pr_{i_1}(\vec{x}_u)\pr_{i_1}(\vec{x}_u)
\end{equation}
where $y_1$ is a new variable. Then from \eqref{blah} and \eqref{plöh} we derive by Identity Rule 
\begin{equation}\label{tähänkin_numero_tasa-arvon_nimissä} y_1 \pr_{i_2}(\vec{x}_{u})\ldots \pr_{i_l}(\vec{x}_u)  \sub  x_{j_1} \ldots x_{j_l}.
\end{equation}
Iterating this procedure $l$ times leads us to a formula
\begin{equation}\label{joo}\bigwedge_{1 \leq k \leq l} \pr_{i_k}(\vec{x}_u) \equiv y_k \ja y_1 \ldots y_l \sub x_{j_1} \ldots x_{j_l}
\end{equation}
where $y_1, \ldots ,y_l$ are pairwise distinct new variables. Let $x_{j_{l+1}}, \ldots ,x_{j_{l'}}$ list $\vec{x} \setminus \{x_{j_1}, \ldots ,x_{j_l}\}$. Repeating Inclusion Introduction for the inclusion atom in \eqref{joo} gives us a formula
\begin{equation}\label{plaah} y_1 \ldots y_{l'} \sub x_{j_1} \ldots x_{j_{l'}}
\end{equation}
where $y_{l+1}, \ldots ,y_{l'}$ are pairwise distinct new variables. Let $\vec{y}$ now denote the sequence $y_1 \ldots y_{l'}$ when 
\begin{equation}\label{jepjep}\bigwedge_{1 \leq k \leq l} \pr_{i_k}(\vec{x}_u) \equiv \pr_k(\vec{y}) \ja \vec{y} \sub x_{j_1} \ldots x_{j_{l'}}
\end{equation}
is the formula obtained from \eqref{joo} by replacing its inclusion atom with \eqref{plaah}.
By Projection and Permutation and Identity Rule we may assume that $\pr_{k}(\vec{y}) = \pr_{k'}(\vec{y})$ if and only if $j_k = j_{k'}$, for $1 \leq k \leq l'$. Analogously to the case $n=0$, we can then order the variables of $\vec{y}$ as a sequence $\vec{x}_{v_{\rm new}}$ of length $|\vec{x}|$ such that $\pr_{j_k}(\vec{x}_{v_{\rm new}})=\pr_k(\vec{y})$, for $1 \leq k \leq l'$. Then 
\begin{equation}\label{jepulis}\bigwedge_{1 \leq k \leq l} \pr_{i_k}(\vec{x}_u) \equiv \pr_{j_k}(\vec{x}_{v_{\rm new}}) \ja  \pr_{j_1}(\vec{x}_{v_{\rm new}}) \ldots \pr_{j_{l'}}(\vec{x}_{v_{\rm new}}) \sub x_{j_1} \ldots x_{j_{l'}}
\end{equation}
is the formula \eqref{jepjep}. By Projection and Permutation we can now deduce $\vec{x}_{v_{\rm new}} \sub \vec{x}$ from the inclusion atom in \eqref{jepulis}. Hence $\vec{x}_{v_{\rm new}}$ is such that $\der \vec{x}_{v_{\rm new}} \sub \vec{x}$ and $\der  \pr_{i_k}(\vec{x}_u) \equiv \pr_{j_k}(\vec{x}_{v_{\rm new}}) $, for $1 \leq k \leq l$. This concludes the case for inclusion.

Assume then that $v_{\rm new}\in V_{m+1}\setminus V_m$ is such that there are $u,w \in V_m$, $u \neq w$, and $\indep{\vec{p}}{\vec{q}}{\vec{r}}  \in\Si$ for which we have added new edges $(u,v_{\rm new})_{x_ix_i}, (w,v_{\rm new})_{x_jx_j}$ to $V_{m+1}$, for $x_i \in \vec{p}\vec{q}$ and $x_j \in \vec{p}\vec{r}$. We will introduce a sequence $\vec{x}_{v_{\rm new}}$ such that $\der \vec{x}_{v_{\rm new}}\sub \vec{x}$, and $\der \pr_{i}(\vec{x}_{u})\equiv \pr_{i}(x_{v_{\rm new}})$ and $\der \pr_{j}(\vec{x}_{w})\equiv \pr_{j}(x_{v_{\rm new}})$, for $x_i \in \vec{p}\vec{q}$ and $x_j \in \vec{p}\vec{r}$. The latter means that 
$$\der \vec{p}_u\vec{q}_u\equiv \vec{p}_{v_{\rm new}}\vec{q}_{v_{\rm new}}\ja  \vec{p}_w\vec{r}_w\equiv \vec{p}_{v_{\rm new}}\vec{r}_{v_{\rm new}}.$$

First of all, we know that $u \sim_{x_kx_k} w$ in $G_m$ for all $x_k \in \vec{p}$. Thus there are vertices $v_0, \ldots ,v_n\in V_m$ and variables $x_{i_0}, \ldots ,x_{i_{n}}$ such that
$$(u,v_0)_{x_kx_{i_0}} , (v_0,v_1)_{x_{i_0}x_{i_1}}, \ldots ,(v_{n-1},v_n)_{x_{i_{n-1}}x_{i_n}},\\(v_n,w)_{x_{i_n}x_k} \in E_m.$$
Hence by the induction assumption and Identity Rule, there are $\vec{x}_u$ and $\vec{x}_w$ such that $\der \vec{x}_u \sub \vec{x}$ and $\der \vec{x}_w \sub \vec{x}$, and  $\der \pr_k(\vec{x}_u)\equiv \pr_k(\vec{x}_w)$, for $x_k \in \vec{p}$. In other words, 
\begin{equation}\label{tästä}\der \vec{p}_u\equiv \vec{p}_w.
\end{equation}

 By Projection and Permutation we first derive 
\begin{equation}\label{label}\vec{p}_u\vec{q}_u \sub \vec{p}\vec{q}
\end{equation}
and
\begin{equation}\label{labello} \vec{p}_w\vec{r}_w\sub \vec{p}\vec{r}
\end{equation} 
from $\vec{x}_u \sub \vec{x}$ and $\vec{x}_w \sub \vec{x}$, respectively. Then we derive 
\begin{equation}\label{labelliino}\vec{p}_u\vec{r}_w\sub \vec{p}\vec{r}
\end{equation} 
from $\vec{p}_u\equiv \vec{p}_w$ and \eqref{labello} by Identity Rule. By Chase Rule we then derive
\begin{equation}\label{huh}\vec{p}_u\vec{q}_u\vec{r}_w\sub \vec{p}\vec{q}\vec{r}
\end{equation}
from $\indep{\vec{p}}{\vec{q}}{\vec{r}}$, \eqref{label} and \eqref{labelliino}.
Now it can be the case that $x_i\in \vec{p}\vec{q}$ and $x_i \in \vec{r}$, but $\pr_i(\vec{x}_u)\neq\pr_i(\vec{x}_w)$. Then we can derive 
\begin{equation}\label{heh}\pr_i(\vec{x}_u)\pr_i(\vec{x}_w) \sub x_ix_i
\end{equation} from \eqref{huh} by Projection and Permutation, and 
\begin{equation}\vec{p}_u\vec{q}_u\vec{r}_w(\pr_i(\vec{x}_u)/\pr_i(\vec{x}_w))\sub \vec{p}\vec{q}\vec{r}
\end{equation}
from \eqref{heh} and \eqref{huh} by Identity Rule. Let now $\vec{r}^*$ be obtained from $\vec{r}_w$ by replacing, for each $x_i\in \vec{p}\vec{q} \cap \vec{r}$, the variable $\pr_i(\vec{x}_w)$ with $\pr_i(\vec{x}_u)$. Iterating the previous derivation gives us then
\begin{equation}\label{numero}\vec{r}^* \equiv \vec{r}_w \ja \vec{p}_u\vec{q}_u\vec{r}^*\sub \vec{p}\vec{q}\vec{r}.
\end{equation}
Let $\vec{s}$ list the variables in $\vec{x}\setminus \vec{p}\vec{q}\vec{r}$. From the inclusion atom in \eqref{numero} we derive by Inclusion Introduction
\begin{equation}\label{joh}\vec{p}_u\vec{q}_u\vec{r}^*\vec{s}^* \sub \vec{p}\vec{q}\vec{r}\vec{s}
\end{equation}
where $\vec{s}^*$ is a sequence of pairwise distinct new variables. Then $\vec{p}_u\vec{q}_u\vec{r}^*\vec{s}^*$ has repetitions at least
where $\vec{p}\vec{q}\vec{r}\vec{s}$ has, and hence we can define $\vec{x}_{v_{\rm{new}}}$ as the sequence of length $|\vec{x}|$ where 
\begin{equation}\vec{p}_u\vec{q}_u\vec{r}^*\vec{s}^*= (\pr_{i_1}(\vec{x}_{v_{\rm{new}}}), \ldots ,\pr_{i_l}(\vec{x}_{v_{\rm{new}}})),
\end{equation} for $\vec{p}\vec{q}\vec{r}\vec{s}=(x_{i_1}, \ldots ,x_{i_l})$. Then $\vec{p}_{v_{\rm new}}\vec{q}_{v_{\rm new}}\vec{r}_{v_{\rm new}}\vec{s}_{v_{\rm new}}=\vec{p}_u\vec{q}_u\vec{r}^*\vec{s}^*$, and we can thus derive
\begin{equation}\vec{x}_{v_{\rm new}} \sub \vec{x}
\end{equation} 
from \eqref{joh} by Projection and Permutation. Moreover,
 \begin{equation}\vec{p}_{v_{\rm new}}\vec{q}_{v_{\rm new}} \equiv \vec{p}_u\vec{q}_u
\end{equation} 
can be derived by Reflexivity, and
\begin{equation}\label{tää}\vec{p}_{v_{\rm new}}\vec{r}_{v_{\rm new}}  \equiv \vec{p}_w\vec{r}_w
\end{equation} 
is derivable since \eqref{tää} is the conjunction of $\vec{p}_u \equiv \vec{p}_w$ in \eqref{tästä} and $\vec{r}^* \equiv \vec{r}_w$ in \eqref{numero}. Hence $\vec{x}_{v_{\rm new}}$ is such that 
$$\der \vec{x}_{v_{\rm new}} \sub \vec{x} \ja \vec{p}_{v_{\rm new}}\vec{q}_{v_{\rm new}} \equiv \vec{p}_u\vec{q}_u \ja \vec{p}_{v_{\rm new}}\vec{r}_{v_{\rm new}}  \equiv \vec{p}_w\vec{r}_w.$$ 
This concludes the case $n=m+1$ and the construction.
\end{itemize}

Assume now first that $\phi$ is $\vec{a} \subseteq \vec{b}$ where $\vec{a}:=x_{i_1}\ldots x_{i_m}$ and $\vec{b}:=x_{j_1}\ldots x_{j_m}$. Then there is a $ w \in V_{\Si ,\phi}$ such that $v\sim_{x_{i_k}x_{j_k}} w$, for $1 \leq k \leq m$. Let $n$ be such that all the witnessing paths are in $G_n$, and let $1 \leq k \leq m$. We first show that
\begin{equation}\label{sama}\der  \pr_{i_k}(\vec{x}_{v})\equiv \pr_{j_k}(\vec{x}_{w}).
\end{equation}
If $w = v $ and $i_k =j_k$, then \eqref{sama} holds by Reflexivity. If $w \neq v $ or $i_k \neq j_k$, then there are vertices $v_0, \ldots ,v_p\in V_n$ and variables $x_{l_0}, \ldots ,x_{l_{p}}$ such that 
$$(v,v_0)_{x_{i_k}x_{l_0}} , (v_0,v_1)_{x_{l_0}x_{l_1}}, \ldots ,(v_{p-1},v_p)_{x_{l_{p-1}}x_{l_p}},(v_p,w)_{x_{l_p}x_{j_k}} \in E_n.$$ 
Then by the previous construction, 
\begin{equation}\der\pr_{i_k}(\vec{x}_{v})\equiv \pr_{l_0}(\vec{x}_{v_0})\ja \ldots \ja\pr_{l_p}(\vec{x}_{v_p})\equiv \pr_{j_k}(\vec{x}_{w})
\end{equation}
when $\der \pr_{i_k}(\vec{x}_{v})\equiv \pr_{j_k}(\vec{x}_{w})$ by Identity Rule. Therefore we conclude that 
\begin{equation}\label{inkl1}\der \vec{a}_v \equiv \vec{b}_w.
\end{equation} 
Since $\der \vec{x}_w \sub \vec{x}$ by the construction, then by Permutation and Projection \begin{equation}\label{inkl2}\der \vec{b}_w \sub \vec{b}.
\end{equation} 
Now $\vec{x}_v = \vec{x}$ as defined in case $1$ of step $n=0$, and therefore $\vec{a}_v=\vec{a}$. Thus we get $\vec{a} \sub \vec{b}$ from \eqref{inkl1} and \eqref{inkl2} using repeatedly Identity Rule. Since no new variables appear in $\vec{a} \sub \vec{b}$, we conclude that $\Si \vdash  \vec{a} \sub \vec{b}$.

Assume then that $\phi$ is $\indep{\vec{a}}{\vec{b}}{\vec{c}}$ when there is a $ v \in V_{\Si ,\phi}$ such that $v^+\sim_{x_ix_i} v$ and $v^-\sim_{x_jx_j} v $ for all $x_i\in\vec{a}\vec{b} $ and $x_j\in\vec{a}\vec{c}$. Analogously to the previous case, we can now find a sequence $\vec{x}_v$ such that
\begin{equation}\label{inkl3}
 \der \vec{x}_v \sub \vec{x}\
\end{equation} 
and
\begin{equation}\label{iden1}
\der \vec{a}_v\vec{b}_v\equiv\vec{a}_{v^+}\vec{b}_{v^+}\ja\vec{a}_v\vec{c}_v \equiv\vec{a}_{v^-}\vec{c}_{v^-}.
\end{equation}
By Projection and Permutation we may deduce 
\begin{equation}\label{inkl4}
\vec{a}_v\vec{b}_v\vec{c}_v \sub \vec{a}\vec{b}\vec{c}
\end{equation} 
from \eqref{inkl3}, and using repeatedly Projection and Permutation and Identity Rule we get
\begin{equation}\vec{a}_{v^+}\vec{b}_{v^+}\vec{c}_{v^-}\sub \vec{a}\vec{b}\vec{c}
\end{equation} from \eqref{iden1} and \eqref{inkl4}. Note that $\vec{a}_{v^+}\vec{b}_{v^+}\vec{c}_{v^-} = \vec{a}\vec{b}\vec{c}^*$ and that we have already derived $\vec{a} \vec{c} \subseteq \vec{a} \vec{c}^*$ and $\indep{\vec{a}}{\vec{b}}{\vec{c}^*}$ with Start Axiom (see case $2$ of step $n=0$). Therefore we can derive $\indep{\vec{a}}{\vec{b}}{\vec{c}}$ with one application of Final Rule. Since no new variables appear in $\indep{\vec{a}}{\vec{b}}{\vec{c}}$, we conclude that $\Si \vdash \indep{\vec{a}}{\vec{b}}{\vec{c}}$.
\end{proof}
By Theorem \ref{sound} and Theorem \ref{complete} we now have the following.
\begin{cor}
Let $\Si \cup \{\phi\}$ be a finite set of conditional independence and inclusion atoms. Then $\Si \vdash \phi$ if and only if $\Si \models \phi$.
\end{cor}
The following example shows how to deduce $\indep{a}{b}{c} \vdash\indep{a}{c}{b}$ and $\indep{a}{b}{cd} \vdash \indep{a}{b}{c}$.
\begin{esim}
\begin{itemize}~\\
\item $\indep{a}{b}{c} \vdash\indep{a}{c}{b}$:
\begin{enumerate}
\item $ab \sub ab' \ja \indep{a}{c}{b'}\ja ab' \sub ab $ (Start Axiom)
\item $ac \sub ac$ (Reflexivity)
\item $\indep{a}{b}{c}\ja ab' \sub ab \ja ac \sub ac   \vdash ab'c \sub abc$ (Chase Rule)
\item $ab'c \sub abc \vdash acb' \sub acb$ (Projection and Permutation)
\item $ab \sub ab' \ja \indep{a}{c}{b'} \ja acb' \sub acb \vdash\indep{a}{c}{b}$ (Final Rule)
\end {enumerate}
\item $\indep{a}{b}{cd} \vdash \indep{a}{b}{c}$:
\begin{enumerate}
\item $ac\sub ac' \ja \indep{a}{b}{c'} \ja ac' \sub ac$ (Start Axiom)
\item $ac'd' \sub acd$ (Inclusion Introduction)
\item $ab \sub ab$ (Reflexivity)
\item $ \indep{a}{b}{cd}\ja ab \sub ab \ja ac'd' \sub acd  \vdash abc'd' \sub abcd$ (Chase Rule)
\item $abc' \sub abc$ (Projection and Permutation)
\item $ac\sub ac' \ja \indep{a}{b}{c'} \ja abc' \sub abc \vdash \indep{a}{b}{c}$ (Final Rule)
\end {enumerate}
\end{itemize}
\end{esim}
Our results shows that for any consequence $\indep{\vec{a}}{\vec{b}}{\vec{c}}$ of $\Si$ there is a deduction starting with an application of Start Axiom and ending with an application of Final Rule.


\bibliographystyle{plain}
\bibliography{biblio}
\end{document}